\documentclass[11pt]{article}
\pdfoutput=1                            
\usepackage[hidelinks]{hyperref}		
\usepackage[margin=1in]{geometry}
\usepackage{rotating}

\def\MODE{3}

\usepackage[utf8]{inputenc}				
\usepackage[english]{babel}				

\usepackage{math}						
\usepackage{enumerate}
\usepackage{cite}
\usepackage{mathrsfs}  
\usepackage{bbding}  
\usepackage{threeparttable}  
\usepackage{mdframed}  
\usepackage[bottom]{footmisc}  
\newmdtheoremenv[innertopmargin=0,innerbottommargin=5,%
innerleftmargin=5,innerrightmargin=5]{fthm}[thm]{Theorem}

\let\oldbibliography\thebibliography
\renewcommand{\thebibliography}[1]{\oldbibliography{#1}
	\setlength{\itemsep}{1pt}} 

\if\MODE3\else\usepackage{generic}\fi
\usepackage{algorithmic}
\usepackage{graphicx}
\usepackage{textcomp}

\graphicspath{{./Figures/}}   

\if\MODE3\usepackage{changepage}\fi   

\newcommand{\ip}[2]{\left\langle #1 ,\, #2 \right\rangle}

\newcommand{\dom}{\mathrm{dom}}

\newcommand{\CPhi}{\mathcal{C}_\Phi}
\newcommand{\RV}{\mathscr{R}(\V)}
\newcommand{\LV}{\mathscr{L}(\V)}

\newcommand{\eemph}[1]{\textit{\textbf{#1}}}
\newcommand{\LLtwoe}{L_{2\textup{e}}}

\usepackage{booktabs}

\if\MODE3\else
\fi

\begin{document}
	
	\title{Generalized Necessary and Sufficient Robust Boundedness Results for Feedback Systems}

	\if\MODE3
	\date{}
	\author{Saman Cyrus\footnote{Department of Electrical and Computer Engineering, University of Wisconsin--Madison. \texttt{cyrus2@wisc.edu}} \and Laurent Lessard\footnote{Department of Mechanical and Industrial Engineering, Northeastern University. \texttt{l.lessard@northeastern.edu}}}
	\else
	\author{Saman Cyrus, \IEEEmembership{Member, IEEE}, Laurent Lessard, \IEEEmembership{Senior Member, IEEE}
	\thanks{This material is based upon work supported by the National Science Foundation under Grant No. 1710892 and 2136317.}
	\thanks{S.~Cyrus is presently at Johnson Controls, Inc. The research related to this work was performed while he was with the Department of Electrical and Computer Engineering at the University of Wisconsin--Madison, Madison, WI 53706. (email: cyrus2@wisc.edu)}
	\thanks{L.~Lessard is with the Department of Mechanical and Industrial Engineering at Northeastern University, Boston, MA 02115. (email: l.lessard@northeastern.edu)}}
	\fi
	
	\maketitle
\begin{abstract}
Classical sufficient conditions for ensuring the robust stability of a dynamical system in feedback with a nonlinearity include passivity, small gain, circle, and conicity theorems. We present a generalized version of these results for arbitrary semi-inner product spaces. Our result is purely algebraic, and holds even when the conventional  discrete or continuous-time causal dynamical systems are replaced by general nonlinear relations, where there need not exist a notion of time. Our result clarifies when the sufficient conditions for robust stability are also necessary, and explains why stronger assumptions such as linearity and time-invariance are typically needed to prove necessity in the conventional dynamical systems setting.
\end{abstract}

\if\MODE3\else
\fi

\section{Introduction}

\if\MODE3
Robust
\else
\IEEEPARstart{R}{obust}
\fi
stability of interconnected systems has been a topic of interest for over 75 years, dating back to the seminal works of Lur'e \cite{lur1944theory}, Zames~\cite{Zames_inputoutput1966_part1,Zames_inputoutput1966_part2}, and Willems~\cite{willems1972dissipative}. The standard input-output setup is illustrated in Fig.~\ref{fig:FeedbackDiagram}, where systems $G$ and $\Phi$ are connected in feedback, and we seek conditions under which we can ensure the stability of the closed-loop map $(u_1,u_2)\to(y_1,y_2)$.

Robust stability results typically assume a known $G$ is interconnected with some unknown, uncertain, or otherwise troublesome $\Phi \in \mathcal{C}_\Phi$, where $\mathcal{C}_\Phi$ is known. Then, if certain conditions on $G$ and $\mathcal{C}_\Phi$ are met, we can ensure that the interconnection of Fig.~\ref{fig:FeedbackDiagram} is stable.

There are many robust stability results in the literature: passivity theory, the small-gain theorem, the circle criterion, graph separation, conic sector theorems, multiplier theory, dissipativity theory, and integral quadratic constraints.\footnote{Detailed references can be found in Section~\ref{sec:related} and Table~\ref{Tab:LiteratureReview}.}
\begin{figure}[t]
	\centering
	\begin{minipage}[T]{0.55\linewidth}
		\centering
		\vspace{0.5em}
		\includegraphics{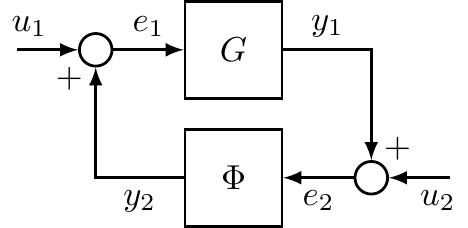}
	\end{minipage}%
	\begin{minipage}[T]{0.45\linewidth}
		\vspace{-1em}
		\begin{subequations}\label{interconnect}
			\begin{align}
			e_1 &= u_1 + y_2 \label{1a}\\
			y_2 &= \Phi e_2 \label{1b}\\
			e_2 &= u_2 + y_1 \label{1c}\\ 
			y_1 &= G e_1 \label{1d}
			\end{align}
		\end{subequations}
	\end{minipage}
	\caption{Feedback interconnection of systems $G$ and $\Phi$.\label{fig:FeedbackDiagram}}
	\vspace{-4mm}
\end{figure}

The reason for the wide variety of robust stability results is that different assumptions can be made about $G$ and $\mathcal{C}_\Phi$.
For example, $G$ and $\Phi$ are typically causal operators on an extended space of time-domain signals such as $\LLtwoe$ or $\ltwoe$. Additionally, $G$ or $\Phi$ may be restricted to be linear, time-invariant, or static. Finally, some results are stated as sufficient conditions while others are both necessary and sufficient.

In spite of their diversity, robust stability results are typically proven using the same elementary properties of inner product spaces. A natural question to ask, which forms the basis of our present work, is whether the multitude of existing results can be viewed as consequences of a purely algebraic result. We answer in the affirmative.
\newpage

\paragraph{Main contribution:}
In Section~\ref{sec2_semi_inner_product_spaces}, we present Theorem~\ref{thm:meta}, a robust boundedness result involving interconnected \emph{relations} over a general \emph{semi-inner product space}. Theorem~\ref{thm:meta} distills the vast literature on robust stability into a simple and purely algebraic result.

In Section~\ref{sec3}, we specialize Theorem~\ref{thm:meta} to $\LLtwoe$ and $\ltwoe$ spaces, which reveals the connections between the algebraic version of the result and notions of \emph{well-posedness}, \emph{causality}, and \emph{stability}. We also explain why stronger assumptions, such as linearity and time-invariance of $G$, are often required in order to achieve both sufficiency and necessity.

\if\MODE3\begin{sidewaystable}\else\begin{table*}[th]\fi
	\centering
	\begin{threeparttable}[b]
		\caption{Literature Review of robust stability results involving two interconnected systems (see Fig.~\ref{fig:FeedbackDiagram}). The first group of rows are sufficient-only results ($\implies$). The second group  are necessary-and-sufficient ($\iff$). For constraints on $G$ and $\Phi\in\CPhi$ we denote linear (L), nonlinear (N), time-varying (TV), time-invariant (TI), static (S), and fading-memory (F). For example, ``LTI'' indicates linear and time-invariant. Symbols $\LLtwoe$ and $\ltwoe$ denote extended spaces ($\Ltwoe$ for both)
		and s.i.p.s. denotes a semi-inner product space.
		In extended spaces, $G$ and $\Phi$ are constrained to be causal. There is no such requirement for s.i.p.s., since there need not exist a notion of time. The final column indicates whether the converse proof direction ($\impliedby$), if applicable, explicitly constructs a worst-case $\Phi$ when the conditions on $G$ are violated. The present work is restricted to static constraints.
		}\label{Tab:LiteratureReview}
	\def\arraystretch{1.2}
	\begin{tabular}{lllccccc}
		\toprule
		\textbf{Reference} & \textbf{Constraint} & \textbf{Result Type}&	\textbf{Space}	& $G$ & $\CPhi$ &\textbf{Direction}  & \textbf{Constructive}\\
		\midrule
		\textbf{Vidyasagar}\cite[\S 6.6.(1,58)]{vidyasagar2002nonlinear} & static & passivity \& small gain	& $\Ltwoe$	&	 N 	& N 		& $\implies$ 	 		& 			\\
		\textbf{Zames}\cite[Thm. 1--3]{Zames_inputoutput1966_part1} 	& static & conic	&$\Ltwoe$ &	N 	& N 		&	$\implies$	 & 	\\
		\textbf{Bridgeman \& Forbes}\cite{Bridgeman2018Comparative} & static & conic				&$\Ltwoe$	& N 		& N 		& $\implies$ 	 &  		\\
		\textbf{Zames}\cite[\S 3--4]{Zames_inputoutput1966_part2} 	& static & circle \& multipliers &	$\LLtwoe$	& LTI 		& NS & $\implies$ 	  		&  			\\
		\textbf{Desoer \& Vidyasagar}\cite{desoer2009feedback} & dynamic & multipliers &$\Ltwoe$ & N &N &$\implies$  &  \\
		\textbf{Teel et al.}\cite{teel1996input} & static & graph separation & $\Ltwoe$ 		& N 		& N 		& $\implies$ 		& 			\\
		\textbf{Willems}\cite{willems1972dissipative} & dynamic & dissipativity & $\Ltwoe$& N& N&$\implies$ & \\
		\textbf{Pfifer \& Seiler}\cite{pfifer2015integral} & dynamic & dissipativity & $\LLtwoe$& LTI &N &$\implies$  &  \\
		\textbf{Megretski \& Rantzer}\cite{megretski1997system} & dynamic & IQC					& $\LLtwoe$	& LTI 		& NS		& $\implies$\tnote{$\dagger$}		 		& 			\\
		%
		\midrule
		\textbf{Vidyasagar}\cite[\S 6.6.(112,126)]{vidyasagar2002nonlinear}	& static & small gain \& circle	&	$\LLtwoe$ &	LTI		& N 		&	$\iff$	 	& Yes			\\
		\textbf{Khong \& van der Schaft}\cite[Thm. 3]{khong2018converse}	& static & passivity \& small gain		&	$\LLtwoe$		&	LTI		& LTV 	&	$\iff$		& No			\\
		\textbf{Zhou et al.}\cite[Thm.~9.1]{zhou1996robust}	& static & small gain			&	$\LLtwoe$			&	LTI		& LTI 		& $\iff$ 		& Yes			\\
		\textbf{Khong \& Kao}\cite[Thm.~1]{khong2021}	& dynamic & IQC		&	$\LLtwoe$				&	LTI		& LTI 		& $\iff$ 		& Yes			\\
		\textbf{Shamma}\cite[Thm. 3.2]{shamma1991necessity}	& static & small gain		&	$\ltwoe$	& 	NF	& NF 	&	$\iff$		& Yes			\\
		\textbf{Cyrus \& Lessard}\cite{cyrus2019unified} 	& static &	conic			& s.i.p.s.		& L & N & $\iff$	 & No 	\\
		\textbf{Present work} 	& static & conic			& s.i.p.s.			& N & N & $\iff$	 & Yes 	\\
		\bottomrule
	\end{tabular}
	\begin{tablenotes}
		\item [$\dagger$] The authors in~\cite{megretski1997system} mention that their sufficient condition for robust stability is also necessary in the sense that a result in the spirit of Lemma~\ref{lem:WorstCaseSignals} holds via a suitable application of the S-lemma~\cite{megretski_treil}.
		\end{tablenotes}
	\end{threeparttable}
\if\MODE3\end{sidewaystable}\else\end{table*}\fi

\subsection{Related work}\label{sec:related}

In Table~\ref{Tab:LiteratureReview}, we provide a summary of existing robust stability results.
In the ``Direction'' column, we distinguish between \emph{sufficient-only} results ($\implies$) and \emph{necessary-and-sufficient} results ($\iff$).

\paragraph{Sufficient results.}

Classical sufficient results include the passivity, small-gain, and circle theorems. These results are mutually related via a loop-shifting transformation~\cite{anderson1972small}, and were generalized to \emph{conic sectors}~\cite{Zames_inputoutput1966_part1,Zames_inputoutput1966_part2,bridgeman2016extended}. 

Beyond conic sector constraints, graph separation~\cite{teel1996input,safonov1980stability} allows for nonlinear constraints, while multiplier theory~\cite{desoer2009feedback}, dissipativity~\cite{willems1972dissipative}, and integral quadratic constraints (IQCs)~\cite{megretski1997system,pfifer2015integral,veenman2016robust} allow for dynamic or time-varying constraints. There have also been several works discussing how these various frameworks are related~\cite{carrasco2019conditions,fu2005integral,seiler2014stability}. In Table~\ref{Tab:LiteratureReview}, we distinguish between \emph{static} constraints (the focus of the present work), and more general \emph{dynamic} constraints, which include multipliers, dissipativity theory, and IQCs.

\paragraph{Necessary and sufficient results.}
When $\Phi$ is assumed to be memoryless (but still possibly time-varying), the classical passivity, small-gain, and circle theorems are only sufficient for robust stability \cite{megretski1993necessary,brockett1966status}.

Finding a robust stability condition that is both sufficient and necessary requires stronger assumptions. The set $\CPhi$ must be broadened to allow dynamic nonlinearities, and we must typically assume that $G$ is linear and time-invariant (LTI).
For example, the passivity and small gain results of Vidyasagar~\cite[\S 6.6(112,126)]{vidyasagar2002nonlinear} and Khong et al.~\cite[Thm.~3]{khong2018converse} assume $G$ is LTI. The small-gain result of Zhou et al.~\cite[Thm.~9.1]{zhou1996robust} and the  converse IQC result of Khong et al.~\cite{khong2021} make the stronger assumption that \emph{both} $G$ and $\Phi$ are LTI. Finally, Shamma's small-gain result~\cite[Thm.~3.2]{shamma1991necessity} holds when both $G$ and $\Phi$ are nonlinear and time-varying, but requires a \emph{fading memory}assumption, which allows the system response to be approximated by that of a linear system.

\subsection{Notation}\label{sec:notation}

\paragraph{Preliminaries.}
The set $\F$ refers to the field of real or complex numbers.
The complex conjugate of $x\in \F$ is $\bar x$ and the conjugate transpose of $X\in \F^{m\times n}$ is $X^*$. We use $\preceq$, $\prec$, $\succ$, $\succeq$ to denote the (semi)definite partial ordering in $\F^{n\times n}$.

\paragraph{Semi-inner products.}
A \emph{semi-inner product space} is a vector space $\V$ over a field $\F$ equipped with a semi-inner product%
\footnote{We use the convention that a semi-inner product is linear in its second argument, so $\ip{x}{ay+bz} = a\ip{x}{y} + b\ip{x}{z}$ for all $x,y,z\in \V$ and $a,b\in \F$. Also, $\ip{x}{y} = \overline{\ip{y}{x}}$.}
$\ip{\cdot}{\cdot}$, which is an inner product whose associated norm is a seminorm. In other words, $\norm{x} \defeq \sqrt{\ip{x}{x}} \geq 0$ for all $x\in \V$, but $\norm{x}=0$ does not imply that $x=0$.

\paragraph{Relations.} A \emph{relation} $R$ on $\V$ is a subset of the product space $R \subseteq \V\times \V$. We write $\RV$ to denote the set of all relations on $\V$. The \eemph{domain} of $R$ is $\dom(R) \defeq \set{x\in \V}{(x,y)\in R \text{ for some } y\in\V}$. For any $x\in\dom(R)$, we write $Rx$ to denote any $y\in\V$ such that $(x,y)\in R$. 

We define $\V^2$ as augmented vectors $\left(\begin{smallmatrix}u_1\\u_2\end{smallmatrix}\right)$ where $u_1,u_2\in\V$. We overload matrix multiplication in~$\V^2$; for any $\xi, \zeta\in\V^2$ and any matrix $N\in\F^{2\times 2}$,
\[
N\xi = \bmat{N_{11} & N_{12} \\ N_{21} & N_{22} }\bmat{\xi_1 \\ \xi_2} \defeq \bmat{ N_{11} \xi_1 + N_{12} \xi_2 \\ N_{21} \xi_1 + N_{22} \xi_2} \in \V^2.
\]
Likewise, inner products in $\V^2$ have the interpretation
\[
\ip{\xi}{\zeta} = \ip{\bmat{\xi_1\\ \xi_2}}{\bmat{\zeta_1\\\zeta_2}} \defeq \ip{\xi_1}{\zeta_1} + \ip{\xi_2}{\zeta_2}.
\]
We omit subscripts when referring to many of the $u_i$, $y_i$, $e_i$ from Fig.~\ref{fig:FeedbackDiagram} at once. For example, $(u,y,e)$ is shorthand for $(u_1,u_2,y_1,y_2,e_1,e_2)$. We also define the following relations, which characterize pairs of consistent signals.
\begin{subequations}\label{eq:relations}
	\begin{align*}
		R_{uy} &\defeq \set{(u,y)\in \V^2\times\V^2 }{\eqref{interconnect} \text{ holds for some }e\in\V^2 },\\
		R_{ue} &\defeq \set{(u,e)\in \V^2\times\V^2 }{\eqref{interconnect} \text{ holds for some }y\in\V^2}.
	\end{align*}
\end{subequations}

\newpage
\section{Results for semi-inner product spaces}
\label{sec2_semi_inner_product_spaces}

Our main result is a robust boundedness theorem defined over a general semi-inner product space.
We consider the setup of Fig.~\ref{fig:FeedbackDiagram}, where $G \in \RV$ and $\Phi \in \CPhi \subseteq \RV$ are (possibly nonlinear) relations.
\begin{fthm}\label{thm:meta}
	Let $\V$ be a semi-inner product space and let $M=M^* \in \F^{2\times 2}$. Suppose $G \in \RV$ and $\CPhi \subseteq \RV$. Consider the three following statements.
	\begin{enumerate}[(i)]
		\itemsep=2mm
		\item\label{thm_it_i} There exists $N=N^* \in \F^{2\times 2}$ satisfying $M+N\prec 0$ such that the following property of $G$ holds.
		\begin{equation}\label{G}
			\ip{ \bmat{G\xi\\ \xi} }{ N \bmat{G\xi\\ \xi} } \ge 0\quad\text{for all } \xi\in\dom(G).
		\end{equation}

		\item\label{thm_it_ii} There exists $\gamma>0$ such that for all $(u,y,e)$, if 
		\begin{equation}\label{M2}
			\ip{ \bmat{e_2\\y_2} }{ M \bmat{e_2\\y_2} } \ge 0
		\end{equation}
		and \eqref{1a}, \eqref{1c}, \eqref{1d} are satisfied, then $\norm{y} \le \gamma \norm{u}$.

		\item\label{thm_it_iii} There exists $\gamma>0$ such that for all $\Phi\in \CPhi$, if
		\begin{equation}\label{Phi}
			\ip{ \bmat{\xi\\\Phi \xi} }{ M \bmat{\xi\\\Phi \xi} } \ge 0
			\quad\text{for all }
			\xi\in\dom(\Phi),
		\end{equation}
		then for all $(u,y)\in R_{uy}$, the following bound holds
		\begin{equation}\label{norm}
			\norm{y} \le \gamma \norm{u}.
		\end{equation}
	\end{enumerate}
	The following equivalences hold:
	\begin{itemize}
		\item Graph separation: $\text{(\ref{thm_it_i})}\iff\text{(\ref{thm_it_ii})}$.
		\item Interpolation: $\text{(\ref{thm_it_ii})}\implies\text{(\ref{thm_it_iii})}$.
	\end{itemize}
\end{fthm}

A pedagogical benefit of Theorem~\ref{thm:meta} is that it splits the robustness result into a \emph{graph separation} statement that concerns $G$ and an \emph{interpolation} statement that concerns $\CPhi$.

The result $\text{(\ref{thm_it_i})}\iff\text{(\ref{thm_it_ii})}$ relates boundedness $G$ in \eqref{G} to boundedness of the closed-loop map when $\Phi$ is replaced by the inequality \eqref{M2}. This graph separation result holds for arbitrary $G$ (any nonlinear relation), and does not depend on $\Phi$ or $\CPhi$.

The result $\text{(\ref{thm_it_ii})}\implies\text{(\ref{thm_it_iii})}$ relates the inequality \eqref{Phi} satisfied by $\Phi$ to the inequality \eqref{M2} satisfied by the inputs and outputs of $\Phi$. Whether or not the converse holds depends on whether the set $\CPhi$ is rich enough to allow interpolation. In other words, given the signals $e_2$ and $y_2$ satisfying~\eqref{M2}, does there necessarily exist a $\Phi\in\CPhi$ such that $y_2 = \Phi e_2$?

Theorem~\ref{thm:meta} is \emph{sufficient} for robust boundedness because it proves (\ref{thm_it_i})$\implies$(\ref{thm_it_iii}). In Section~\ref{sec:achievingnecessity}, we show that with suitable assumptions about $G$ and $\CPhi$, we can satisfy the interpolation requirement and therefore make the result necessary as well.

Since Theorem~\ref{thm:meta} is expressed using a general semi-inner product space, it holds even when $G$ is not a dynamical system but rather a general nonlinear relation. So there need not exist a notion of time.
We make a few additional remarks.

\begin{rem}
	Equation~\eqref{norm} can be stated in terms of $(u,e)$ instead of $(u,y)$. Specifically,
	\eqref{norm} holds for all $(u,y)\in R_{uy}$ if and only if there exists some $\bar\gamma > 0$ such that $\norm{e} \le \bar\gamma\norm{u}$ holds for all $(u,e)\in R_{ue}$.
\end{rem}

\begin{rem}
	In Item (\ref{thm_it_i}), we can equivalently replace $N$ by $-M-\varepsilon I$ and modify the statement preceding~\eqref{G} to: ``There exists some $\varepsilon > 0$ such that $G$ satisfies \eqref{G}''. We chose the form with $M$ and $N$ for aesthetic reasons.
\end{rem}

\begin{rem}
	Theorem~\ref{thm:meta} can be generalized to $G \in \mathcal{R}(\V^n,\V^m)$ (the set of relations on $\V^n \times \V^m$) and $\Phi \in \CPhi \subseteq \mathcal{R}(\V^m,\V^n)$. Here, $M,N \in \F^{(m+n)\times(m+n)}$ would be block $2\times 2$ matrices. 
\end{rem}

\subsection{Proof of sufficiency for Theorem~\ref{thm:meta}} \label{sec:SufficiencyProof}

We begin by showing that (\ref{thm_it_i})$\implies$(\ref{thm_it_ii})$\implies$(\ref{thm_it_iii}).
This proof is similar to \cite[Thm.~1]{mccourt2010control}. Pick any $(u,y,e)$ such that \eqref{1a}, \eqref{1c}, \eqref{1d}, and \eqref{M2} are satisfied. 
Let $\xi=e_1$ in~\eqref{G}. Using \eqref{interconnect} to eliminate $e_1$, $e_2$, Equations~\eqref{G} and~\eqref{M2} become: $\ip{ \sbmat{y_1\\ u_1+y_2}\! }{ N \sbmat{y_1\\u_1+y_2} } \ge 0$ and
$\ip{ \sbmat{u_2+y_1\\y_2}\! }{ M \sbmat{u_2+y_1\\y_2} } \ge 0$.
Summing these two inequalities and collecting terms, we obtain
\[\ip{\bmat{y_1\\y_2}}{(M+N)\bmat{y_1 \\ y_2}}
	+2\ip{\bmat{y_1\\y_2}}{\bmat{N_{12}&M_{11}\\N_{22}&M_{21}}\bmat{u_1\\u_2}}
	+ \ip{\bmat{u_1\\u_2}}{\bmat{N_{22}&0\\0&M_{11}}\bmat{u_1\\u_2} }
	\ge 0.\]
Since $M+N \prec 0$ by assumption, There exists $\eta > 0$ such that $M + N
\preceq -\eta I$. Applying this inequality together with Cauchy--Schwarz\footnote{A proof of the Cauchy--Schwarz inequality for general semi-inner product spaces may be found in~\cite[\S 1.4]{conway}.}, we get $-\eta\norm{y}^2 + 2r \norm{y} \norm{u} + q \norm{u}^2 \ge 0$, where 
$r \defeq \normmm{ \sbmat{N_{12}&M_{11}\\N_{22}&M_{21}} }$
and
$q \defeq \normmm{ \sbmat{N_{22}&0\\0&M_{11}} }$
are standard spectral norms. Dividing by $\eta$ and completing the square, we can rewrite the last inequality as
$\bigl(\norm{y}-\tfrac{r}{\eta}\norm{u}\bigr)^2 \leq \tfrac{r^2+\eta q}{\eta^2} \norm{u}^2$, which can be rearranged to establish~(\ref{thm_it_ii}) with
$\gamma = \frac{1}{\eta}\bigl( r + \sqrt{r^2 + \eta q} \bigr)$.

To prove (\ref{thm_it_iii}), consider some $\Phi \in \CPhi$ for which~\eqref{Phi} holds. Next, pick $(u,y) \in R_{uy}$ so that there exists $(u,y,e)$ satisfying \eqref{interconnect}. In particular, \eqref{1b} holds, so setting $\xi = e_2$ in \eqref{Phi}, we obtain~\eqref{M2} and the rest of the proof is the same as above. \qedhere

\subsection{Necessity of graph separation in Theorem~\ref{thm:meta}}
\label{sec:necessity}

A popular approach for proving (\ref{thm_it_i})$\impliedby$(\ref{thm_it_ii}) is to use a lossless S-lemma as in~\cite[Thm.~3]{khong2018converse} and \cite{megretski1997system}. However, the S-lemma \cite{megretski_treil,yakubovich1997s} comes with a drawback: the set of signals $(u,y,e)$ that satisfy the loop equations~\eqref{1a}, \eqref{1c}, \eqref{1d} must be a \emph{subspace}, which requires for example that $G$ be linear. If we assume $G$ is linear, we can prove (\ref{thm_it_i})$\impliedby$(\ref{thm_it_ii}) by adapting the S-lemma for inner product spaces due to Hestenes~\cite[Thm.~7.1,~p.~354]{hestenes} and using a technique similar to that used in \cite{khong2018converse}. Details of this approach may be found in \cite{cyrus2020dissertation,cyrus2019unified}.

The linearity assumption on $G$ can be dropped entirely if we adopt a different proof approach. To this effect, we will prove the contrapositive $\lnot$(\ref{thm_it_i})$\implies$$\lnot$(\ref{thm_it_ii}) by directly constructing signals $(y,u,e)$ that violate the boundedness condition when (\ref{thm_it_i}) fails to hold. Unlike the S-lemma, this approach does not require linearity of $G$ and has the benefit of being constructive, so it produces \emph{worst-case} signals $(u,y,e)$.

\newpage
\begin{lem}[worst-case signals]\label{lem:WorstCaseSignals}
	Consider the setting of Theorem~\ref{thm:meta}. Suppose that for any $N$ satisfying $M+N \prec 0$, there exists $\xi\in\dom(G)$ such that
	$
	\ip{\sbmat{G\xi \\ \xi}}{N\sbmat{G\xi \\ \xi}} < 0
	$.
	Then, for all $\gamma > 0$, there exists $(u,y,e)$ such that:
	\begin{enumerate}
		\itemsep=1mm
		\item\label{lemit1} Equations \eqref{1a}, \eqref{1c}, and \eqref{1d} hold.
		\item\label{lemit2} $
		\ip{\bmat{\vphantom{G}e_2 \\ y_2}}{M\bmat{\vphantom{G}e_2 \\ y_2}} \geq 0$.
		\item\label{lemit3} $\norm{y} > \gamma \norm{u}$.
	\end{enumerate}
\end{lem}
\begin{proof}
See Appendix~\ref{sec:proofoftwoimpliesone} for a detailed proof.\hfill
\end{proof}
The implication (\ref{thm_it_i})$\impliedby$(\ref{thm_it_ii}) of Theorem~\ref{thm:meta} now directly follows from Lemma~\ref{lem:WorstCaseSignals}.

\subsection{Necessity of interpolation in Theorem~\ref{thm:meta}}\label{sec:achievingnecessity}

The implications $\text{(\ref{thm_it_i})}\iff\text{(\ref{thm_it_ii})}\implies\text{(\ref{thm_it_iii})}$ of Theorem~\ref{thm:meta} hold with great generality. However, the missing implication $\text{(\ref{thm_it_ii})}\impliedby\text{(\ref{thm_it_iii})}$ does not hold in general, because it depends on the choice of $G$ and $\CPhi$. If $\CPhi$ is insufficiently expressive, there may not exist a $\Phi\in\CPhi$ that interpolates the closed-loop signals found in $\text{(\ref{thm_it_ii})}$. We now explore some special cases for which the missing implication holds; in other words, there exists a $\Phi$ that interpolates the closed-loop signals.

\begin{defn} 
We say that a pair $(G,\CPhi)$ is \emph{interpolable} if such a choice implies that (\ref{thm_it_ii})$\impliedby$(\ref{thm_it_iii}) in Theorem~\ref{thm:meta}.
\end{defn}

We now describe simple scenarios in which interpolability is guaranteed for the general semi-inner product setting. First, we make the trivial observation that if $\Phi$ is unconstrained, interpolation is always possible.

\begin{prop}[unconstrained case] \label{thm:singleton_extension}
	If $\CPhi = \RV$, then the pair $(G,\CPhi)$ is interpolable for any $G\in\RV$.
\end{prop}

\begin{proof}
If $\CPhi=\RV$, then \eqref{M2} and \eqref{Phi} are equivalent, as we can choose the singleton relation $\Phi = \{(e_2,y_2)\}$.\hfill
\end{proof}

Proposition~\ref{thm:singleton_extension} is not particularly satisfying because it requires the use of a singleton relation $\Phi$. A more interesting case is when we require that $\dom(\Phi) = \V$.\footnote{Relations $\Phi \in \RV$ that satisfy $\dom(\Phi) = \V$ are known as \emph{serial} or \emph{left-total}. They are also called \emph{multi-valued functions}.} Our second result states that interpolability holds for the set of \emph{linear} relations.

\begin{defn}[linear relation]
	Let $\V$ be a semi-inner product space over a field $\F$. Let $x_1,x_2,y_1,y_2 \in \V$ and $\alpha_1,\alpha_2 \in \F$. A relation $R \in \RV$ is \emph{linear} if for all $(x_1,y_1)\in R$ and $(x_2,y_2)\in R$, we have
$(\alpha_1 x_1 + \alpha_2 x_2, \alpha_1 y_1 + \alpha_2 y_2) \in R$. We let $\LV \subseteq \RV$ denote the set of all linear relations.
\end{defn}

\begin{thm}[linear case] \label{thm:linear_extension}
	If $\CPhi = \LV$, then the pair $(G,\CPhi)$ is interpolable for any $G \in \RV$.
\end{thm}

\begin{proof}
We explicitly construct a worst-case $\Phi \in \LV$. See Appendix~\ref{sec:proofoflinearextension} for a detailed proof.\hfill
\end{proof}

Proposition~\ref{thm:singleton_extension} and Theorem~\ref{thm:linear_extension} both provide conditions that ensure necessity of Theorem~\ref{thm:meta}. In both cases, there are no constraints on $G$; it could be nonlinear, for example.

\section{Specialization to extended spaces}
\label{sec3}

The most common application of robust stability is when $(y,u,e)$ are time-domain signals belonging to an extended space such as $\LLtwoe$ or $\ltwoe$~\cite{zhou1996robust}. This forces us to deal with well-posedness, causality, and stability.

\paragraph{Well-posedness.}
Assuming $G$ and $\Phi$ are relations, as we do in Theorem~\ref{thm:meta}, is not unprecedented in the literature \cite{Zames_inputoutput1966_part1,vidyasagar2002nonlinear,Schaft2017L2,khong2018converse,safonov1980stability}.
This ensures the closed-loop relations $R_{uy}$ and $R_{ue}$ are always well-defined, but they may be \emph{empty}. When $G$ and $\Phi$ are assumed to be operators instead of relations, then well-posedness must either be assumed or proved. Specifically, we need an assurance of the existence and uniqueness of solutions $e$ and $y$ for all choices of $u$.

\paragraph{Causality.}
When working in extended spaces such as $\Ltwoe$, a common assumption is that $G$ and $\Phi$ are causal operators \cite{zhou1996robust,megretski1997system,Zames_inputoutput1966_part2,vidyasagar2002nonlinear,desoer2009feedback,teel1996graphs}. Then, a useful fact is that a well-posed interconnection of causal maps is causal~\cite[Prop.~1.2.14]{Schaft2017L2}, so the closed-loop map will be causal.

\paragraph{Stability.}
The goal when working with time-domain signals is typically to prove stability. With Theorem~\ref{thm:meta}, we prove boundedness of the closed-loop map, i.e., $\norm{y} \leq \gamma \norm{u}$, and therefore input-output stability.

To specialize Theorem~\ref{thm:meta} to extended spaces, set $\V = \Ltwoe$ and use the semi-inner product $\ip{\cdot}{\cdot}_T$ defined by projecting both signals onto $[0,T]$ and applying the $\Ltwoe$ inner product. Then, use the fact that if $H$ is a causal map,  $\norm{Hx}\leq \gamma\norm{x}$ for all $x\in\Ltwo$ if and only if $\norm{Hx}_T\leq \gamma\norm{x}_T$ for all $x$ and $T$ (see, for example, \cite[Lem.~6.2.11]{vidyasagar2002nonlinear}).

Different choices of the matrices $M$ and $N$ allow the representation of different cones. For example, we can represent different flavors of passivity (input-strict, output-strict, extended), small-gain results, the circle criterion, and other conic sectors that allow $G$ or $\Phi$ to be unbounded/unstable.

To illustrate these various transformations, consider for example the classical passivity result by Vidyasagar, which is a sufficient-only result, and may be found in~\cite[Thm.~6.7.43]{vidyasagar2002nonlinear}. 
\begin{thm}[Vidyasagar]\label{thm:vidyasagar}
	Consider the system 
	\[
	\left\{
	\begin{array}{lr}
	e_1  = u_1 - y_2, & y_1 = G e_1\\
	e_2 = u_2 + y_1, & y_2 = \Phi e_2
	\end{array}
	\right.
	\]
	Suppose there exist constants $\epsilon_1$, $\epsilon_2$, $\delta_1$, $\delta_2$ such that for all $\xi\in\ltwoe$ and for all $T\ge 0$
	\begin{subequations}\label{eq:passivityeq}
		\begin{align}
		\ip{\xi}{G\xi}_T &\ge \epsilon_1 \norm{\xi}_T^2 + \delta_1\norm{G\xi}_T^2,\\
		\ip{\xi}{\Phi \xi}_T &\ge \epsilon_2 \norm{\xi}_T^2 + \delta_2\norm{\Phi \xi}_T^2.
		\end{align}
	\end{subequations}
	Then the system is $\ltwo$-stable if $\delta_1+ \epsilon_2>0$ and $\delta_2+\epsilon_1>0$.		
\end{thm}

Theorem~\ref{thm:vidyasagar} uses a negative sign convention and is expressed in discrete time. To match \ref{thm:meta}, let $\Phi\mapsto -\Phi$ in Theorem~\ref{thm:vidyasagar} and compare~\eqref{G} and~\eqref{Phi} to~\eqref{eq:passivityeq}, which yields
\[
N =\bmat{-\delta_1 & \frac{1}{2}\\ \frac{1}{2} & -\epsilon_1}
\quad\text{and}\quad
M = \bmat{-\epsilon_2& -\frac{1}{2}\\ -\frac{1}{2} & -\delta_2 }.
\]
In Theorem \ref{thm:meta}, we require $M+N\prec 0$; thus $\delta_1+ \epsilon_2>0$ and $\delta_2+\epsilon_1>0$, which recovers Theorem~\ref{thm:vidyasagar}. A similar approach can be used to recover all the results from Table~\ref{Tab:LiteratureReview} involving static constraints. For a detailed proof, see~\cite{cyrus2020dissertation}.

Unlike Theorem~\ref{thm:meta}, Theorem~\ref{thm:linear_extension} does not specialize as nicely to extended spaces. In particular, the construction of a worst-case $\Phi$ from Lemma~\ref{lem:extension} will not, in general, be causal. In order to achieve interpolability, one must typically make additional assumptions, such as $G$ and $\Phi$ being linear and time-invariant (see Table~\ref{Tab:LiteratureReview}). In such cases, a worst-case $\Phi$ can be chosen as a static gain cascaded with a time delay \cite[\S 6.6.(112,126)]{vidyasagar2002nonlinear}.

\section{Conclusion}

We studied robust stability results involving a plant $G$ connected with a nonlinearity $\Phi$ belonging to a conic sector, e.g.  passivity, small-gain, circle criterion, conicity, or extended conicity. Our goal was to distill the vast literature on this topic and state the most general and unified results possible.

Looking beyond the scope of this paper, it would be interesting to see if our semi-inner product framework could be used to recover results involving dynamic constraints (dissipativity, multiplier theory, integral quadratic constraints).

\section{Acknowledgments}
The authors would like to thank R.~Boczar, \mbox{L.~Bridgeman}, R.~Brockett, S.~Z.~Khong, A.~Packard, A.~Rantzer, P.~Seiler, A.~van~der~Schaft, B.~Van~Scoy, and M.~Vidyasagar for helpful discussions and comments through various stages of this work.

\newpage
\if\MODE3
\appendix
\else
\appendices
\fi
\section{Proofs for semi-inner product spaces}\label{sec:appendix1}

\subsection{Proof of Lemma \ref{lem:WorstCaseSignals}}
\label{sec:proofoftwoimpliesone}
\begin{proof}
The result is trivial or vacuous if $M$ is semidefinite, so we will assume $M$ is indefinite, writing
\begin{equation}\label{eq:M}
	M = P^* J P,
\end{equation}
where $J \defeq \diag(1,-1)$ and $P \in \F^{2\times 2}$ is invertible.
Pick any $0 < \varepsilon < 1$ and let $N = -M - \varepsilon P^* P$. By assumption, we can choose some $\xi \in \dom(G)$ such that
\begin{equation}\label{eq:xi_ineq}
	\ip{\bmat{G\xi \\ \xi}}{N \bmat{G\xi \\ \xi}} < 0.
\end{equation}
Now pick $e_1 = \xi$, $y_1 = G\xi$, and
\[
\bmat{e_2\\y_2} = \bmat{y_1 \\ e_1} + \bmat{u_2 \\ -u_1}\text{, with }	\bmat{u_2 \\ -u_1} = \varepsilon\, P^{-1}JP \bmat{y_1 \\ e_1}.
\]
By construction, this choice satisfies Item \ref{lemit1} of Lemma~\ref{lem:WorstCaseSignals}.
Substituting our choice of $N$ into~\eqref{eq:xi_ineq}, we obtain
\begin{equation}\label{eq:M_ineq}
\ip{\bmat{y_1\\e_1}}{M\bmat{y_1\\e_1}} > -\varepsilon\, \normmmm{P\bmat{y_1 \\ e_1}}^2.
\end{equation}
Substituting the definitions of $u_1$, $u_2$, $e_2$, $y_2$, $M$ in terms of $y_1$, $e_1$, $\varepsilon$, $P$, $J$, and using the inequality~\eqref{eq:M_ineq}, we have
\begin{align*}
	\ip{\bmat{e_2\\y_2}}{M\bmat{e_2\\y_2}} &= 
	\ip{(I+\varepsilon P^{-1}J P)\bmat{y_1 \\ e_1}}{M(I+\varepsilon P^{-1}J P)\bmat{y_1 \\ e_1} }\\
	&= 
	\ip{(I+\varepsilon J )P\bmat{y_1 \\ e_1}}{P^{-*}MP^{-1}(I+\varepsilon J)P \bmat{y_1 \\ e_1} }\\
	&\overset{\eqref{eq:M}}{=} 
	\ip{(I+\varepsilon J )P\bmat{y_1 \\ e_1}}{J(I+\varepsilon J)P \bmat{y_1 \\ e_1} }\\
	&=\ip{P\bmat{y_1\\e_1}}{\bigl( (1+\varepsilon^2)J + 2\varepsilon I\bigr)P\bmat{y_1\\e_1}} \\
	&\overset{\eqref{eq:M}}{=} (1+\varepsilon^2)\ip{\bmat{y_1\\e_1}}{M\bmat{y_1\\e_1}} +2\varepsilon\normmmm{P\bmat{y_1\\e_1}}^2 \\
		&\overset{\eqref{eq:M_ineq}}{>} \varepsilon\bigl(1-\varepsilon^2\bigr) \normmmm{P\bmat{y_1\\e_1}}^2 \geq 0,
\end{align*}
which verifies Item~\ref{lemit2} of Lemma~\ref{lem:WorstCaseSignals}. Finally, we have:
\begin{align*}
\normmmm{\bmat{u_1 \\ u_2}} &= \normmmm{\bmat{u_2 \\ -u_1}}
= \normmmm{\varepsilon P^{-1}JP\bmat{y_1 \\ e_1}}
 \leq \varepsilon  \kappa \normmmm{\bmat{y_1\\e_1}},
\end{align*}
where $\kappa \defeq \normm{P^{-1} J P}>0$. Applying the triangle inequality,
\begin{align*}
	\norm{u} = \normmmm{\bmat{u_1 \\ u_2}} \leq  \varepsilon\kappa \normmmm{\bmat{y_1\\e_1}}
	= \varepsilon\kappa\normmmm{\bmat{y_1\\y_2+u_1}} 
& \leq \varepsilon\kappa\left( \normmmm{\bmat{y_1 \\ y_2}} + \normmmm{\bmat{0\\u_1}} \right) \leq \varepsilon\kappa \bigl( \norm{y} + \norm{u} \bigr).
\end{align*}
Rearranging, we obtain $\norm{y} \geq \frac{1-\varepsilon\kappa}{\varepsilon\kappa}\norm{u}$. Since $\varepsilon$ can be chosen arbitrarily small, we can make the bound $\gamma$ arbitrarily large in Item~\ref{lemit3} of Lemma~\ref{lem:WorstCaseSignals}, thus completing the proof.\hfill
\end{proof}

\newpage
\subsection{Proof of Theorem~\ref{thm:linear_extension}}\label{sec:proofoflinearextension}

We begin by proving that a pair of points satisfying a quadratic constraint can be extended to a linear relation that satisfies the quadratic constraint everywhere.

\begin{lem}[extension lemma]\label{lem:extension} Let $\V$ be a semi-inner product space and let $M=M^*\in\F^{2\times 2}$. Suppose $e,y \in \V$ satisfy
\begin{equation}\label{eq:eyM}
	\ip{ \bmat{e \\ y}}{M\bmat{e \\ y}} \geq 0.
\end{equation}
There exists $\Phi \in \LV$ such that:
\begin{enumerate}
	\itemsep=1mm
	\item $(e,y) \in \Phi$.
	\item $\ip{ \bmat{x \\ \Phi x}}{M\bmat{x \\ \Phi x} } \geq 0$ for all $x \in \dom(\Phi)$.
\end{enumerate}
Moreover, if $\norm{e} > 0$, we can construct $\Phi$ that is a \emph{linear function}, with $\dom(\Phi) = \V$.
\end{lem}

Using Lemma~\ref{lem:extension}, we can prove Theorem~\ref{thm:linear_extension} by contradiction. Indeed, if Item~(\ref{thm_it_ii}) of Theorem~\ref{thm:meta} fails, then for any $\gamma > 0$, there exist $e_2,y_2 \in \V$ such that~\eqref{M2} and \eqref{1a}, \eqref{1c}, \eqref{1d} hold, with $\norm{y} > \gamma \norm{u}$. Applying Lemma~\ref{lem:extension} to the pair $(e_2,y_2)$, we can produce $\Phi \in \LV \subseteq \CPhi$ such that \eqref{Phi} holds, and thus \eqref{1b} holds, $(u,y) \in R_{uy}$, and therefore Item~(\ref{thm_it_iii}) of Theorem~\ref{thm:meta} fails, as required. All that remains is to prove Lemma~\ref{lem:extension}.

\begin{proof}
	We begin by considering some special cases.
	\paragraph{Special case with $\norm{e}=0$.} Here, $\ip{e}{y}=0$ by Cauchy--Schwarz. If $\norm{y}=0$, define $\Phi \defeq \set{(z,x)}{ \norm{z}=\norm{x}=0 }$. This is a degenerate case. If $\norm{y}>0$ instead, we have by assumption that
	$
		M_{22} \norm{y}^2 =
		\ip{ \bmat{e \\ y}}{M\bmat{e \\ y}} \geq 0
		$.
	Therefore, $M_{22} \geq 0$. Define $\Phi = \set{(z,x)}{\norm{z}=0}$. Roughly, $\Phi$ is the linear relation whose graph is a \emph{vertical line}.

	\paragraph{Special case with $\norm{e}>0$ and $\norm{y}=0$.}
	As in the previous case, we must have $\ip{e}{y}=0$. By assumption,
	$
		M_{11} \norm{e}^2 = \ip{\bmat{e \\ y}}{M\bmat{e \\ y}} \geq 0
	$.
	So, $M_{11} \geq 0$. Let
	$\Phi x \defeq 0$, so
	$
		\ip{\bmat{x \\ \Phi x}}{M\bmat{x \\ \Phi x}} = M_{11} \norm{x}^2 \geq 0\quad\text{for all }x\in\V
	$.
	Henceforth, we will assume that $\norm{e}>0$ and $\norm{y}>0$. Define the normalized vectors $\hat e \defeq \frac{e}{\norm{e}}$ and $\hat y \defeq \frac{y}{\norm{y}}$. Also define $\rho \defeq \langle \hat e, \hat y\rangle$. Note that by Cauchy--Schwarz, we have $|\rho|\leq 1$.\footnote{Recall that in general, inner products are elements of $\F$, so $\rho$ may be a complex number.}

	\paragraph{Special case: $|\rho| = 1$.}
	Define $\Phi x = \rho \frac{\norm{y}}{\norm{e}} x$ and obtain:
	\[
	\ip{\bmat{x \\ \Phi x}}{M\bmat{x \\ \Phi x}} =
	\frac{\norm{x}^2}{\norm{e}^2}
	\ip{\bmat{e \\ y}}{M\bmat{e \\ y}} \geq 0.
	\]

	\paragraph{General case: $|\rho|<1$.}
	
	Due to \eqref{eq:eyM}, we have $M\nprec 0$. So there must exist some $\eta \in \F$ such that $\sbmat{1 \\ \eta}^*M\sbmat{1 \\ \eta} \geq 0$.
	For any $x\in\V$, apply the projection theorem to decompose $x = x_{ey} + x_\perp$, where $x_{ey}$ is a linear combination of $\hat e$ and $\hat y$ and $x_\perp$ is orthogonal to both $\hat e$ and $\hat y$. This yields
	$$
	\begin{aligned}
		x_{ey} &\defeq \left( \frac{\langle \hat e, x\rangle - \rho \langle \hat y, x \rangle}{1-|\rho|^2} \right) \hat e + \left( \frac{\langle \hat y, x\rangle - \bar\rho \langle \hat e, x \rangle}{1-|\rho|^2} \right) \hat y,\\
		x_\perp &\defeq x - x_{ey}.
	\end{aligned}
	$$
	Note that if $x=e$, we have $e_{ey}=e$ and $e_\perp=0$.
	We also have $\norm{x}^2 = \norm{x_{ey}}^2 + \norm{x_\perp}^2$.
	Define the unit vectors
	\[
	\hat e_\perp \defeq \frac{\hat y - \rho \hat e}{\sqrt{1-|\rho|^2}}
	\quad\text{and}\quad
	\hat y_\perp \defeq \frac{\bar\rho\hat y - \hat e}{\sqrt{1-|\rho|^2}}.
	\]
	The vectors $\hat e_\perp$ and $\hat y_\perp$ are orthogonal to $\hat e$ and $\hat y$, respectively.
	Write $M_{12} = |M_{12}|e^{i\varphi}$ (polar decomposition). Since $M_{21} = \overline{M_{12}}$, we have: $e^{-2i\varphi}M_{12} = M_{21}$.
	Finally, define $\Phi$ as
	$$
	\Phi x \defeq \frac{\norm{y}}{\norm{e}}\Bigl(\langle \hat e,x_{ey}\rangle\hat y + e^{-2i\varphi} \langle \hat e_\perp, x_{ey}\rangle \hat y_\perp \Bigr) + \eta\, x_\perp.
	$$
	The function $\Phi$ is linear 
	and using the fact that $e_{ey}=e$ and $e_\perp=0$, it follows that $\Phi e = y$.
	Moreover, one can check that
	$
	\norm{\Phi x_{ey}} = \frac{\norm{y}}{\norm{e}}\norm{x_{ey}}
	$
	and
	$
	\mathrm{Re}\Bigl( M_{12} \langle x_{ey}, \Phi x_{ey} \rangle \Bigr) = \frac{\norm{y}}{\norm{e}}\norm{x_{ey}}^2\,\mathrm{Re}(M_{12} \rho).
	$
	Thus, $\bmat{x \\ \Phi x}= \bmat{x_{ey} \\ \Phi x_{ey}}+\bmat{x_\perp \\ \eta x_\perp }$ and $\left\langle \bmat{x \\ \Phi x},M\bmat{x \\ \Phi x} \right\rangle$
	\begin{equation*}
		= \left\langle \bmat{x_{ey} \\ \Phi x_{ey}},M\bmat{x_{ey} \\ \Phi x_{ey}} \right\rangle\\
		+\left\langle \bmat{x_\perp \\ \eta x_\perp},M\bmat{x_\perp \\ \eta x_\perp}\right\rangle.
	\end{equation*}
	The first term simplifies to: $\left\langle \bmat{x_{ey} \\ \Phi x_{ey}},M\bmat{x_{ey} \\ \Phi x_{ey}} \right\rangle$
	\begin{align*}
		& = M_{11}\|x_{ey}\|^2 + 2\,\mathrm{Re}\Bigl( M_{12} \langle x_{ey}, \Phi x_{ey} \rangle \Bigr)
		+ M_{22}\|\Phi x_{ey}\|^2 \\
		&= \|x_{ey}\|^2 \left( M_{11} + 2 \,\mathrm{Re}( M_{12} \rho ) \frac{\|y\|}{\|e\|} + M_{22} \frac{\|y\|^2}{\|e\|^2} \right) \\
		&= \frac{\|x_{ey}\|^2}{\|e\|^2} \left( M_{11} \|e\|^2 + 2\,\mathrm{Re}\left( M_{12} \langle e,y \rangle\right)  + M_{22}\|y\|^2 \right) \\
		&= \frac{\|x_{ey}\|^2}{\|e\|^2}\left\langle \bmat{e \\ y},M\bmat{e \\ y} \right\rangle 
		\geq 0.
	\end{align*}
	The second term simplifies to
	\[
	\left\langle \bmat{x_\perp \\ \eta x_\perp},M\bmat{x_\perp \\ \eta x_\perp}\right\rangle = \|x_\perp\|^2 \bmat{1 \\ \eta}^*M\bmat{1 \\ \eta} \geq 0.
	\]
	Therefore, we have $\left\langle \bmat{x \\ \Phi x},M\bmat{x \\ \Phi x} \right\rangle\geq 0$, as required.\hfill
\end{proof}


\bibliographystyle{abbrv}
\if\MODE3\small\bibliography{refs}
\else\bibliography{refs}\fi

\end{document}